\documentclass[11pt,a4paper]{article}

\usepackage[english]{babel}
\usepackage[utf8x]{inputenc}
\usepackage[a4paper,top=3cm,bottom=2cm,left=3cm,right=3cm,marginparwidth=1.75cm]{geometry}
\usepackage{authblk}
\usepackage{amsmath}
\usepackage{graphicx}
\usepackage[colorinlistoftodos]{todonotes}
\usepackage[colorlinks=true, allcolors=blue]{hyperref}
\usepackage{amsfonts}
\usepackage{tikz-cd}
\usepackage{tikz}
\usepackage{pst-node}
\usepackage{mathtools}
\usepackage{amssymb}
\usepackage{amsmath,amsfonts,amssymb,amsthm,epsfig,epstopdf,titling,url,array}
\usepackage{comment}
\usepackage{enumitem}

\usepackage[T2A]{fontenc}

\makeatletter
\newtheorem*{rep@theorem}{\rep@title}
\newcommand{\newreptheorem}[2]{%
\newenvironment{rep#1}[1]{%
 \def\rep@title{#2 \ref{##1}}%
 \begin{rep@theorem}}%
 {\end{rep@theorem}}}
\makeatother

\theoremstyle{plain}
\newtheorem{thm}{Theorem}[section]

\newtheorem{lem}[thm]{Lemma}

\newtheorem{cor}[thm]{Corollary}
\theoremstyle{definition}
\newtheorem{rem}[thm]{Remark}
\newtheorem{defn}[thm]{Definition}
\newtheorem{exmp}[thm]{Example}

\theoremstyle{definition}

\theoremstyle{remark}

\numberwithin{equation}{section}

\newcommand{\SO}{\operatorname{SO}}
\newcommand{\GL}{\operatorname{GL}}

\newcommand\da{Diophantine approximation}

\newcommand{\R}{{\mathbb{R}}}
\newcommand{\Z}{{\mathbb{Z}}}
\newcommand{\bs}{{\mathbb{S}}}

\newcommand{\vp}{{\bf p}}
\newcommand{\vq}{{\bf q}}
\newcommand{\vv}{{\bf v}}

\newcommand{\ignore}[1]{{}}

\newcommand {\commdk}[1]   {\textcolor{red}{\bf #1}}

\newcommand\eq[2]{
\begin{equation}
\label{eq:#1}
{#2}
\end{equation}
}
\newcommand{\equ}[1]{\eqref{eq:#1}}

\title{Critical loci of convex domains in the plane}
\author{Dmitry Kleinbock, Anurag Rao, Srinivasan Sathiamurthy}

\newcommand{\Addresses}{{
  \bigskip
  \footnotesize

\textsc{Dmitry Kleinbock,   Brandeis University, Waltham MA
02454-9110}, {\tt kleinboc@brandeis.edu}

  \medskip

  \textsc{Anurag Rao,  Brandeis University, Waltham MA
02454-9110}, {\tt anrg@brandeis.edu}

  \medskip

  \textsc{Srinivasan Sathiamurthy,  Lexington MA 02420}, {\tt srinivasansathiamurthy@gmail.com}

}}

\begin{document}

\maketitle
\begin{abstract}
Let   $K$ be a bounded   convex domain in $\R^2$ symmetric about the origin. The {\sl critical locus} of $K$ is  defined to be the (non-empty compact) set of lattices $\Lambda$ in $\R^2$ of smallest possible covolume such that $\Lambda \cap K= \lbrace 0\rbrace$. These are classical objects in geometry of numbers; yet all previously known examples of critical loci were either finite sets or finite unions of closed curves.
In this paper we  give a new construction which, in particular,  furnishes examples of domains  having critical locus of arbitrary Hausdorff dimension between $0$ and $1$.
\end{abstract}

\section{Introduction}


\ignore{$H$ and $K$ denote open, bounded, symmetric, convex domains in $\R^2$.}

Counting lattice points in convex symmetric domains is a classical problem, which dates back to Gauss and belongs to 
 \textit{geometry of numbers},   the branch of number theory that studies number-theoretical problems by the use of geometric methods. Geometry of numbers in its proper sense was pioneered by Minkowski, see \cite{Mi} or the books \cite{ca, GL} for a comprehensive introduction to the subject.

A \textsl{lattice\/} $\Lambda$ in $\R^n$ is the set of all integer linear combinations of $n$ linearly independent vectors $\vv_1,\dots,\vv_n\in\R^n$;  in other words 
\begin{equation}\label{lambda}
\Lambda = \Z\vv_1 + \cdots + \Z\vv_n = g\Z^n,
\end{equation}
where $g\in\GL_n(\Z)$ is the matrix with column vectors $\vv_1,\dots,\vv_n\in\R^n$. The most natural example to consider is the grid of points $\Z^n \subset \R^n$, which is generated by the standard basis of $\R^n$ and corresponds to $g = I_n$,  the $n\times n$ identity matrix. Note that $g\Z^n = \Z^n$ if and only if $g\in\GL_n(\Z)$, where the latter stands for the group of invertible $n\times n$ matrices with integer entries. Consequently, the space $X_n$ of lattices in $\R^n$ is isomorphic, as a $\GL_n(\R)$-space, to the quotient space $\GL_n(\R)/\GL_n(\Z)$.

For any $\Lambda \in X_n$ of the form \eqref{lambda} we define its {\it covolume\/} $d(\Lambda)$ as the volume of the parallelepiped spanned by $\vv_1,\dots,\vv_n$. Clearly it does not depend on the choice of the generating set and is equal to the absolute value of the determinant of $g$.

\smallskip
Let now $K$ be a bounded   convex domain in $\R^n$ symmetric about the origin; denote by $V(K)$   the volume of $K$.

\begin{defn}
A lattice $\Lambda\subset \R^n$ is called $K$-\textsl{admissible} if $\Lambda \cap K= \lbrace 0\rbrace.$ To $K$ as above we associate a real number $\Delta(K)$, called the {\it critical determinant} of $K$, given by \begin{equation}
    \Delta(K):= \inf \left\lbrace d(\Lambda) : \Lambda \text{ is } K\text{-admissible}\right\rbrace
\end{equation}
\end{defn}

Perhaps one the most fundamental results in geometry of numbers is   Minkowski's Convex Body Theorem, see e.g.\ \cite[\S III.2.2]{ca}, which states that for $K$ as above, any lattice in $\R^n$ with covolume less than $V(K)/2^n$   must contain a  point of $K$ distinct from $0$. In other words, for any such $K$ the critical determinant $ \Delta(K) $ is positive, and, moreover, one has 
\begin{equation*}\label{minkimplication}
 \Delta(K) \ge \frac{V(K)}{2^n}.
\end{equation*}

This motivates the problem of exhibiting $K$-admissible lattices with the smallest covolume; those are called $K$-critical.



\begin{defn}
A lattice $\Lambda$ in $\R^n$ is called $K$-\textsl{critical} if it is a $K$-admissible lattice with $d(\Lambda)=\Delta(K).$ The set of $K$-critical lattices is denoted by $\mathcal{L}(K)$ and is called the \textsl{critical locus} of $K$.
\end{defn}

Since admissibility is preserved by taking limits, this set is non-empty by a sequential compactness argument due to Mahler, which can be found e.g.\ in \cite[\S V.4.2]{ca}.
The critical locus $\mathcal{L}(K)$ can be thought of as a subset of $X_n$, and we will  endow it with the  topology induced from $\operatorname{GL}_n(\R)/\operatorname{GL}_n(\Z)$.
With this topology it is compact, again by Mahler's argument.

It is worthwhile to point out  that  $K$-critical lattices are in one-to-one correspondence with the densest lattice packings of $\R^n$ by translates of  $K$. 
Indeed, suppose that $\Lambda$ is $K$-admissible. 
Then it is easy to see that the collection of sets $$\lbrace 2\vv + K:  \vv \in \Lambda\rbrace$$ is pairwise disjoint.
And, conversely, if such a collection of sets is pairwise disjoint for some $\Lambda$, then $\Lambda$ must be $K$-admissible. 
Minimizing the covolume of $\Lambda$ over all admissible lattices thus corresponds to maximizing the relative area covered by the collection of the above sets.
Another motivation for studying critical loci of convex symmetric domains comes from \da, see \S\ref{da} for more detail. 

\smallskip
From now on let us restrict our attention to the case $n=2$. Critical loci of planar domains have been systematically studied by Mahler in a series of papers written in the 1940s,  
see {Section} \ref{exandprelim} for a list of examples. 
Perhaps the most straightforward is the case when $$K = D:= \{(x,y): x^2 + y^2 < 1\}$$   is the unit disc: then  one has   $\Delta(D)=\sqrt{3}/2$, and 
\begin{equation}\label{disc}
 \mathcal{L}(D) =    \left\lbrace   k\left[{\begin{array}{cc}
   1 & 1/2 \\
   0 & \sqrt{3}/2 \\
  \end{array}}\right] \Z^2: k \in \SO(2) \right\rbrace,
\end{equation}
which is homeomorphic to $\bs^1$.

However all examples of sets $\mathcal{L}(K)$  previously constructed by Maher and others were either finite sets or finite unions of closed curves, and there does not exist a precise description of compact subsets of $X_2$ which can arise as critical loci of convex symmetric bounded domains of $\R^2$. 
(The situation for $n>2$ is even further from being understood.)
For instance, after studying Examples  \ref{firstex}--\ref{lastex} below, one may pose the following natural question:
$$
\textit{Is it possible for some } K \textit{ to have critical locus homeomorphic to a Cantor set?}
$$

We give a construction that shows that this is indeed possible. In particular, we prove

\begin{thm}\label{disc}
Each non-empty closed subset of $\mathcal{L}(D)$, where $D\subset \R^2$ is the unit disc, is the critical locus of some convex symmetric domain $K\supset D$.
\end{thm}

In particular, one can choose a closed subset of  $\mathcal{L}(D)$ homeomorphic to the middle-third Cantor set, or to any other closed fractal subset of $[0,1]$.

\smallskip

More generally, instead of $D$ we can take any  \textsl{irreducible} 
strictly convex\footnote{The strict convexity of $H$ means that the line joining two distinct points on $\partial H$ cannot intersect $\partial H$ at any other points.}  symmetric domain.
The concept of irreducibility, introduced by Mahler in \cite{ma1}, is as follows:
\begin{defn}\label{irred-defn}
A  convex symmetric bounded domain $K$ in $\R^2$ is said to be \textsl{irreducible} if each  convex symmetric bounded $H\subsetneqq K$ has $\Delta(H) < \Delta(K)$.
$K$ is said to be \textsl{reducible} if it is not irreducible, that is, if there exists $H{\subsetneqq} K$ with  $\Delta(H)=\Delta(K)$.
\end{defn}

Here and hereafter $H,K$ will always denote  convex symmetric bounded  domains in $\R^2$.  We remark that Mahler actually defines a stronger notion of irreducibility; namely, according to his definition $K$ is irreducible if all bounded symmetric \textit{star domains} $S\subsetneqq K$ have $\Delta(S) <\Delta(K)$. However for convex domains $K$ it is equivalent to the one given in Definition \ref{irred-defn}; this follows from \cite[Theorem 1]{ma2}, see a footnote to Lemma \ref{irredcritset} below.

\smallskip

Among the examples in Section  \ref{exandprelim} below, the square $[-1.1]^2$, the disc $D$  and their images under linear transformations are irreducible.
{A far more interesting example of an irreducible domain is the curvilinear octagon described in \cite[Section 12]{ma3}, see Remark~\ref{diagram-mahler}.}

\smallskip
Mahler studied irreducible domains and proved a number of fundamental results illustrating their importance.
For example, 
the  theorem stated below shows that irreducible domains are ubiquitous in the following sense:
\begin{thm}[\cite{ma2}, Theorem $1$]\label{containsirred}
Every $K\subset \R^2$ contains an irreducible $H$ with \linebreak $\Delta(H) = \Delta(K)$.
\end{thm}

In Section \ref{exandprelim}, capitalizing on this and other results of Mahler, we prove
\begin{thm}
\label{irredlocus}
If $K$ is irreducible but not a parallelogram, then 
there is a continuous map $\phi: [0,1] \to 
X_2$ which descends to a homeomorphism 
$\bs^1\to\mathcal{L}(K)$.
\end{thm}
In the case of a parallelogram, as made explicit in Example \ref{square}, the critical locus is topologically the wedge of two circles.
\ignore{With the above theorems in mind, 
\commdk{I don't think this paragraph belongs to the introduction. Better to say that in Section 2 we will explain why Theorem \ref{containsirred} implies Corollary \ref{cor-to-conversify}, and then have Proof of Corollary \ref{cor-to-conversify} after the list of examples there.}}

\begin{cor}\label{cor-to-conversify}
If $K$ is not a parallelogram, $\mathcal{L}(K)$ is a closed subset of an embedded circle $\mathcal{L}(H)$, the critical locus of some irreducible non-parallelogram $H$.
\end{cor}
This simple corollary of Theorems  \ref{containsirred} and  \ref{irredlocus} is proved in 
Section \ref{exandprelim}, and our main theorem is a partial converse to it:
\begin{thm}\label{all-closed}
If $H\subset \R^2$ is strictly convex and irreducible, so that $\mathcal{L}(H)$ is an embedded circle inside $X_2$, 
 then each non-empty closed subset of $\mathcal{L}(H)$ is the critical locus of some domain $K\supset H$.
\end{thm}
 
In view of the irreducibility of $D$, Theorem \ref{disc} is clearly a special case of Theorem \ref{all-closed}.
Note that the above theorem does not hold for parallelograms.
Indeed, if $H$ is a parallelogram, one can consider any   subset of $\mathcal{L}(H)$ homeomorphic to the transverse intersection of two line segments; clearly it could never be embedded in a circle, and therefore, as Corollary \ref{cor-to-conversify} would imply, could not be a critical locus of any $K$.
At the end of Section  \ref{proof-of-theorem} we discuss the further possibility of the strict convexity assumption being removed for non-parallelogram irreducible domains.

As for the rest of the article, Section  \ref{exandprelim} gives the basic theorems on irreducible domains and critical lattices and examples to motivate the study,
and   Section \ref{proof-of-theorem} contains the proof of Theorem \ref{all-closed}.



\ignore{\subsection{Motivation and Further Remarks}
Given the relationship between $K$ and various critical lattices, we considered the possibilities of critical $\Lambda$ over all $K$.
\newline \newline
\noindent To make our question simpler, since $d(\Lambda) = \Delta(K)$, we can renormalize the lattices and instead only consider the space of normalized two dimensional lattices, $SL_2(\mathbb{R})/SL_2(\mathbb{Z})$.
\newline \newline
\noindent This however was too far of a step so we tried to find some parts of our set in the space of lattices, namely, any set homeomorphic to a non-empty, compact subset of $\mathbb{R}$.
}

\section{Preliminaries and examples}\label{exandprelim}

The following fundamental and intuitive theorem on admissible and critical lattices is taken from \cite[\S V.8.3]{ca} where one may also find a formal proof.
It shows that critical lattices are realised by inscribed parallelograms in the domains.
\begin{thm}\label{threepairs}
Let $\Lambda$ be $K$-critical, and let $C$ be the boundary of $K$. 
Then one can find three pairs of points $\pm \vp_1, \pm \vp_2, \pm \vp_3$ of the lattice on $C$.
Moreover these three points can be chosen such that
\begin{equation}\label{inscirbedhexagon}
\vp_1 + \vp_2 = \vp_3
\end{equation}
and any two vectors among  $\vp_1,\ \vp_2,\ \vp_3$ form a basis of $\Lambda.$

Conversely, if $\vp_1, \vp_2, \vp_3$ satisfying 
 {\eqref{inscirbedhexagon}} are on $C$,
then the lattice generated by $\vp_1$ and $\vp_2$ is $K$-admissible. 
Furthermore no additional (excluding the six above) point $\vp_4$ of $\Lambda$ is on $C$ unless $K$ is a parallelogram.
\end{thm}

In light of this theorem, we may discuss the critical loci for the following examples of domains $K\subset \R^2$. 

\begin{exmp}\label{firstex} \rm
As was mentioned in the introduction, when $K = D$, the unit disc, one has that $\Delta(D)=\sqrt{3}/2$, and the set of critical lattices is given by \eqref{disc}.
\begin{equation*}\label{diagram-disc}
    \begin{tikzpicture}[scale=1.5, baseline=(current  bounding  box.center)]
    \draw [<->](-1.5,0) -- (1.5,0);
    \draw [<->] (0,-1.5) -- (0,1.5);
    
    \draw [thick, blue] (1,0) to [out=90, in=0] (0,1);
    \draw [thick, blue] (0,1) to [out=180, in=90] (-1,0);
    \draw [thick, blue] (-1,0) to [out=270, in=180] (0,-1);
    \draw [thick, blue] (0,-1) to [out=0, in=270] (1,0);
    
    \fill [red] (1,0) circle[radius=0.05];
    \node [above] at (1.1,0) {\tiny $\vp$};
    \fill [red] (1/2,1.73205/2) circle[radius=0.05];
    \node [above] at (1/2,1.73205/2) {\tiny $\vq$};
    \fill [red] (-1/2,1.73205/2) circle[radius=0.05];
    \node [above] at (-1/2,1.73205/2) {\tiny $\vq-\vp$};
    \fill [red] (-1,0) circle[radius=0.05];
    \fill [red] (-1/2,-1.73205/2) circle[radius=0.05];
    \fill [red] (1/2,-1.73205/2) circle[radius=0.05];
    \fill [red] (0,0) circle[radius=0.05];
    
        \node [below] at (0,-1.6) { \footnotesize Points $\vp$ and $\vq$ generate a critical lattice.};
    \node [below] at (0,-1.83) {\footnotesize All other critical lattices are obtained by rotating this one.};
    \end{tikzpicture}
\end{equation*}
More precisely, the map $\phi:\R \to X_2$
given by
\begin{equation*}
    \phi(t) := \left[{\begin{array}{cc}
   \cos t & -\sin t \\
   \sin t & \cos t \\
  \end{array}}\right] \left[{\begin{array}{cc}
   1 & 1/2 \\
   0 & \sqrt{3}/2 \\
  \end{array}}\right]  \Z^2
\end{equation*}
descends to a homeomorphism $\R/(\frac{\pi \Z}{3}) \simeq \mathcal{L}(D)$.

This is an example of an irreducible domain.
For if $H$ is a subset of $D$ with $\Delta(H) =\Delta(D)$, then all $D$-critical lattices are $H$-admissible and thus also $H$-critical.
Then the first part of {Theorem} \ref{threepairs} shows that each $D$-critical lattice $\phi(t)$ above must contain three pairs of points on $\partial H$.
Since $H\subset D$, those three pairs of points must coincide exactly with the set $\phi(t) \cap \partial D$, thus showing that the boundaries of $H$ and $D$ coincide.

\end{exmp}
\begin{exmp}\label{square} \rm
When $K$ is the square with side-length $2$ and   sides perpendicular to the coordinate axes, Theorem \ref{threepairs} again shows that the critical lattices are given by 
\begin{equation}
    \left\lbrace \left[{\begin{array}{cc}
   1 & t \\
   0 & 1 \\
  \end{array}}\right]  \Z^2: t\in \R\right\rbrace \bigcup  \left\lbrace \left[{\begin{array}{cc}
   1 & 0 \\
   t & 1 \\
  \end{array}}\right] \Z^2: t\in \R\right\rbrace,
\end{equation}
which is topologically a wedge of two circles.
In this case too, $K$ is irreducible by the same application of Theorem \ref{threepairs} as in the case of the unit disc. 
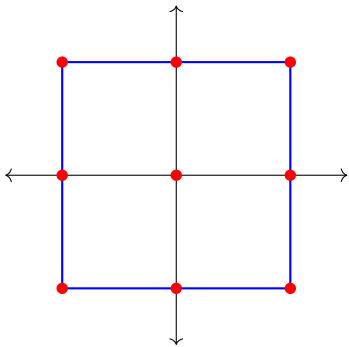
\begin{equation*}\label{diagram-square}
    \begin{tikzpicture}[scale=1.5, baseline=(current  bounding  box.center)]
    \draw [<->](-1.5,0) -- (1.5,0);
    \draw [<->] (0,-1.5) -- (0,1.5);
    
    \draw [thick, blue] (1,-1) -- (1,1);
    \draw [blue, thick] (1,1) -- (-1,1);
    \draw [blue, thick] (-1,1) -- (-1,-1);
    \draw [blue, thick] (-1,-1) -- (1,-1);
    
    \fill [red] (1,0) circle[radius=0.05];
    \fill [red] (-1,0) circle[radius=0.05];
    \fill [red] (1,1) circle[radius=0.05];
    \fill [red] (-1,-1) circle[radius=0.05];
    \fill [red] (-1,1) circle[radius=0.05];
    \fill [red] (1,-1) circle[radius=0.05];
    \fill [red] (0,1) circle[radius=0.05];
    \fill [red] (0,-1) circle[radius=0.05];
    \fill [red] (0,0) circle[radius=0.05];

        \node [below] at (0,-1.6) {\footnotesize Shearing the standard lattice along each axis gives rise to all the critical lattices.};
    \end{tikzpicture}
\end{equation*}
\end{exmp}

\begin{exmp}\label{hexex} \rm
When $K$ is a hexagon, there is exactly one critical lattice. 
This lattice is spanned by the two vectors on the midpoints of adjacent sides.
Moreover, in this case one has $V(K) = 4\Delta(K)$.
This follows as a corollary of Theorem \ref{threepairs}, and a proof can be found in \cite[\S V.8.4]{ca}.
It is also known \cite[Theorem 5]{ma3} that when $K$ is a $2n$-gon with $n\geq 3$, there can only be a finite number of critical lattices. 

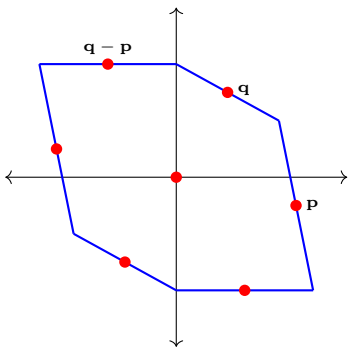
\begin{equation*}\label{diagram-hex}
    \begin{tikzpicture}[scale=1.5, baseline=(current  bounding  box.center)]
    \draw [<->](-1.5,0) -- (1.5,0);
    \draw [<->] (0,-1.5) -- (0,1.5);
    
    \draw [blue, thick] (1.2,-1) -- (0.9,0.5);
    \draw [blue, thick] (-1.2,1) -- (-0.9,-0.5);
    
    \draw [blue, thick] (0.9, 0.5) -- (0,1);
    \draw [blue, thick] (-0.9, -0.5) -- (0,-1);
    
    \draw [blue, thick] (0, 1.0) -- (-1.2,1);
    \draw [blue, thick] (0, -1.0) -- (1.2,-1);
    
    \fill [red] (1.05,-0.25) circle[radius=0.05];
    \node [right] at (1.05,-0.25) {\tiny $\vp$};
    \fill [red] (-1.05,0.25) circle[radius=0.05];
    
    \fill [red] (0.45,0.75) circle[radius=0.05];
    \node [right] at (0.45,0.77) {\tiny $\vq$};
    \fill [red] (-0.45,-0.75) circle[radius=0.05];
    
    \fill [red] (-0.6,1) circle[radius=0.05];
    \node [above] at (-0.6,1) {\tiny $\vq-\vp$};
    \fill [red] (0.6,-1) circle[radius=0.05];
    
    \fill [red] (0,0) circle[radius=0.05];
    
    \node [below] at (0,-1.6) {\footnotesize The lattice generated by $\vp$ and $\vq$ is the only critical lattice.};
    \end{tikzpicture}
\end{equation*}
\end{exmp}
\begin{exmp}\label{lastex} \rm
If $K$ is the unit ball of the $L^p$ norm on $\R^2$ ($1 < p < \infty, \ p\neq 2$), then, 
depending on $p$, the critical locus comprises either of one or two lattices.
For a historical account and a proof of this (Minkowski's conjecture), see the paper \cite{ggm}.
\end{exmp}
\begin{rem} \rm
If $H = gK$ for some  $g\in \operatorname{GL}_2(\R)$, then the critical loci of $H$ and $K$ are related by $\mathcal{L}(H)=g\mathcal{L}(K)$.
Also, irreducibility is preserved under transformation by $g$.
\end{rem}

We now collect the results needed to make Theorem \ref{irredlocus} more precise. 
Moreover, the parameterization of the critical locus arising will be crucial to our result.

\begin{lem}[\cite{ma2}, Lemma $6$]\label{configuration-of-points}
{Suppose that} $K$ is not a parallelogram, $\Lambda$ is $K$-critical, and let $\vp_i: i = {1, \dots, 6}$ be the points of $\Lambda$ 
contained in  $C = \partial K$, labelled in a counter-clockwise order.
Let $A_i$ denote open segment of $C$ between $\vp_i$ and $\vp_{i+1}$.
If $\Lambda'$ is another $K$-critical lattice distinct from $\Lambda$, then each $A_i$ contains exactly one point of $\Lambda'$.
\end{lem}

\begin{lem}[\cite{ma2}, Lemma $9$\footnote{Even though the statement of Lemma \ref{irredcritset} is taken almost verbatim from Mahler's paper, our meaning is changed in light of the different definitions of {irreducibility}. Logically speaking, one can use \cite[Theorem 1]{ma2}  to show that the two definitions of irreducibility are equivalent.}]\label{irredcritset}
Assume $K$ is not a parallelogram and is irreducible. Then for each $\vp \in C = \partial K$  there is exactly one critical lattice of $K$ containing $\vp$.
\end{lem}

\begin{proof}[Proof of Theorem  \ref{irredlocus}]
Fix $\Lambda$, any critical lattice for $K$, and let $C$ denote the boundary of $K$.
Let $\vp_i$ denote the six points of $C \cap \Lambda$ oriented counter-clockwise.
By convexity, one can parameterize the closed segment of $C$ between $\vp_1$ and $\vp_2$ by a continuous map $\vp:[0,1] \to \R^2$.

From Lemma \ref{irredcritset} above, we see that each $\vp(t)$ belongs to a unique $K$-critical lattice $\Lambda_t$.
Let $\vq(t)$ denote the point of $\Lambda_t \cap C$ coming after  $\vp(t)$ going counter-clockwise in angle.
We claim that the function $\vq$ is continuous. 
If not,
then for some $t$  we would find a neighborhood $U$ of $\vq(t)$ with a converging sequence $t_n\to t$ but $\vq(t_n) \notin U$.
Continuity of $\vp$ implies $\vp(t_n) \to \vp(t)$.
Without loss of generality, we can assume $t_n <t$ so that applying Lemma \ref{configuration-of-points} to a lattice $\Lambda$ containing a point in $U\cap C$ between $\vq(t)$ and each $\vq(t_n)$ gives a contradiction.
Thus we can set $$\phi(t) = \left[
   \vp(t)\ \vq(t) \right]\Z^2,$$ which, in light of Lemma \ref{configuration-of-points}, descends to the required homeomorphism.
\end{proof}

\begin{proof}[Proof of Corollary \ref{cor-to-conversify}]\label{proof-of-1.6}
Suppose that $K$ is not a parallelogram.
By Theorem \ref{containsirred}, $K$ contains an irreducible $H$ with $\Delta(K)=\Delta(H)$.
It follows that we have the containment $\mathcal{L}(K) \subset \mathcal{L}(H)$.
Moreover we claim that $H$ cannot be a parallelogram.

For if this were the case, as made explicit in Example \ref{square} above, $\Delta(H)$ would have to equal $\frac{V(H)}{4}$.
However, we also have that $V(H) < V(K)$ and that $\frac{V(K)}{4} \leq \Delta(K)$.
Together these give 
$$\Delta(H) = \frac{V(H)}{4} < \frac{V(K)}{4} \leq \Delta(K),
$$
a contradiction.

Since $H$ is not a parallelogram, Theorem \ref{irredlocus} applies and we are done.
\end{proof}

\begin{rem}\label{actually-C1} \rm
Mahler proved (cf.\ Theorem $3$ in \cite{ma3}) that for such $K$, the boundary $C$ is a $\mathcal{C}^1$ submanifold of $\R^2$. Thus the preceding proof can be modified to show that $\mathcal{L}(K)$ is a $\mathcal{C}^1$ submanifold of the space of lattices.
\end{rem}

\section{Proof of {Theorem} \ref{all-closed}}\label{proof-of-theorem}
The proof of our theorem will be based on the following simple observation:
\begin{lem}\label{observation}
If $H\subset K$  and one of the $H$-critical lattices is also $K$-admissible, then $\mathcal{L}(K)$ is exactly the set of $K$-admissible lattices in $\mathcal{L}(H)$.
\end{lem}
\begin{proof}
Since $H \subset K$, we have that $\Delta(H) \leq \Delta(H)$.
Further, the existence of an $H$-critical and $K$-admissible lattice shows that we have the equality $\Delta(H) = \Delta(K)$.

First, let $\Lambda$ be $K$-critical.
It is then also $H$-admissible with $d(\Lambda) = \Delta(K) = \Delta(H)$. 
Thus it is $H$-critical and, by definition, also $K$-admissible.

For the other containment, let $\Lambda \in \mathcal{L}(H)$ be $K$-admissible.
Since $\Delta(H) = \Delta(K)$, $\Lambda$ must in fact be $K$-critical.
\end{proof}
\begin{proof}[Proof of Theorem \ref{all-closed}]
Take $H$, not a parallelogram and irreducible, and say $Z\subset \mathcal{L}(H)$ is  non-empty and compact.
Let $\Lambda \in Z$.
Let $C$ denote the boundary of $H$, and, as before, let us label the six points of $\Lambda \cap C$ as  $\vp_1,\dots, \vp_6$ (ordered by angle).
Let $\vp(t)$ denote a continuous parameterization of the segment $C$ from $\vp_1$ to $\vp_2$ by the interval $[0,1]$ (see diagram \eqref{tent-picture} below), let $\Lambda_t$ denote the unique critical lattice containing $\vp(t)$,  and let $\vq(t)$ denote the point of $\Lambda_t \cap C$ coming after $\vp(t)$ (going counter-clockwise). 
\begin{equation}\label{tent-picture}
    \begin{tikzpicture}[scale=2, baseline=(current  bounding  box.center)]
    \draw [<->](-1.5,0) -- (1.5,0);
    \draw [<->] (0,-1.5) -- (0,1.5);
    
    \draw [thick, blue] (1,0) to [out=90, in=0] (0,1);
    \draw [thick, blue] (0,1) to [out=180, in=90] (-1,0);
    \draw [thick, blue] (-1,0) to [out=270, in=180] (0,-1);
    \draw [thick, blue] (0,-1) to [out=0, in=270] (1,0);
    
    \fill [red] (1,0) circle[radius=0.03];
    \fill [red] (1/2,1.73205/2) circle[radius=0.03];
    \fill [red] (-1/2,1.73205/2) circle[radius=0.03];
    \fill [red] (-1,0) circle[radius=0.03];
    \fill [red] (-1/2,-1.73205/2) circle[radius=0.03];
    \fill [red] (1/2,-1.73205/2) circle[radius=0.03];
    \fill [red] (0,0) circle[radius=0.03];
    
    \node [above] at (1.15,-0.2) {\tiny $\vp(0)$};
    \node [above] at (1.1/2,1.83205/2) {\tiny $\vp(1) = \vq(0)$};
    \node [above] at (-1/2,1.73205/2) {\tiny $\vq(1)$};
    
    \fill [orange] (0.988,0.156) circle[radius=0.03];
    \node [right] at (0.988,0.156) {\tiny $\vp(a_i)$};
    
    \fill [orange] (0.707,0.707) circle[radius=0.03];
    \node [right] at (0.707,0.747) {\tiny $\vp(b_i)$};
    
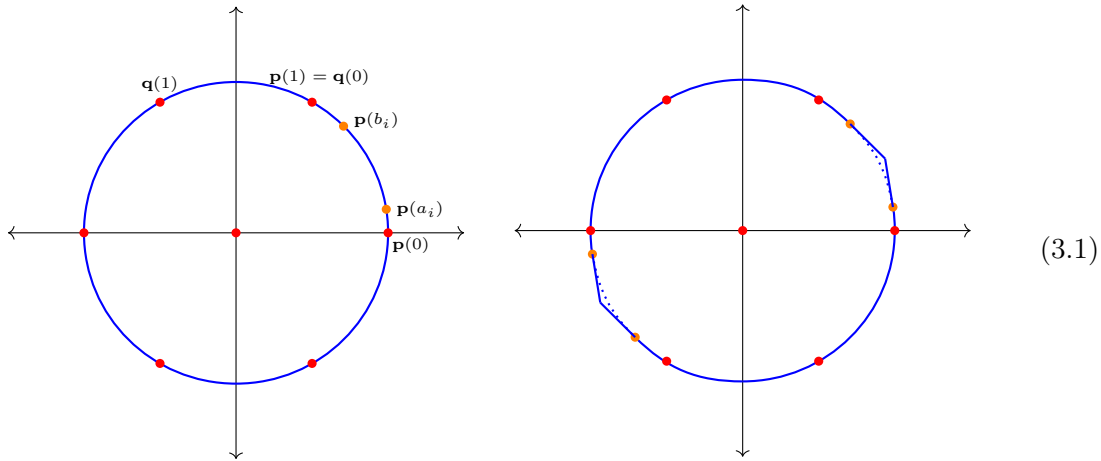
        \node [below] at (0,-1.6) {\footnotesize An illustration when $H$ is a disc.};
    \end{tikzpicture}
    \quad
    \begin{tikzpicture}[scale=2, baseline=(current  bounding  box.center)]
    \draw [<->](-1.5,0) -- (1.5,0);
    \draw [<->] (0,-1.5) -- (0,1.5);
    
    \draw [thick, blue] (1,0) to [out=90, in=280] (0.988,0.156);
    \draw [thick, blue] (-1,0) to [out=270, in=100] (-0.988,-0.156);
    \draw [thick, blue] (0.707,0.707) to [out=135, in=0] (0,1);
    \draw [thick, blue] (-0.707,-0.707) to [out=315, in=180] (0,-1);
    \draw [thick, blue] (0,1) to [out=180, in=90] (-1,0);
    \draw [thick, blue] (0,-1) to [out=0, in=270] (1,0);
    \draw [dotted, thick, blue] (0.988,0.156) to [out=100 , in=315] (0.707,0.707);
    \draw [dotted, thick, blue] (-0.988,-0.156) to [out=280 , in=135] (-0.707,-0.707);
    
    \fill [red] (1,0) circle[radius=0.03];
    \fill [red] (1/2,1.73205/2) circle[radius=0.03];
    \fill [red] (-1/2,1.73205/2) circle[radius=0.03];
    \fill [red] (-1,0) circle[radius=0.03];
    \fill [red] (-1/2,-1.73205/2) circle[radius=0.03];
    \fill [red] (1/2,-1.73205/2) circle[radius=0.03];
    \fill [red] (0,0) circle[radius=0.03];
    
    \fill [orange] (0.988,0.156) circle[radius=0.03];
    
    \fill [orange] (0.707,0.707) circle[radius=0.03];
    
    \draw [thick, blue] (0.707, 0.707) -- (0.9366,0.4771);
    \draw [thick, blue] (0.988,0.156) -- (0.9366,0.4771);
    
    \fill [orange] (-0.988,-0.156) circle[radius=0.03];
    \fill [orange] (-0.707,-0.707) circle[radius=0.03];
    
    \draw [thick, blue] (-0.988,-0.156) -- (-0.9366,-0.4771);
    \draw [thick, blue] (-0.707, -0.707) -- (-0.9366,-0.4771);
    
        \node [below] at (0,-1.6) {\footnotesize The union of $H$ and two  curvilinear regions.};
    \end{tikzpicture}
\end{equation}

We use the parameterization $\phi(t) := [\vp(t)\ \vq(t)]\Z^2$ for $\mathcal{L}(H)$. We can now assume that our non-empty, closed set is the image under $\phi$ of some compact $Q \subset [0,1]$  with $\lbrace 0, 1 \rbrace \subset Q$.
Set 
\begin{equation*}\label{tenthere}
    [0,1]\smallsetminus Q = \bigsqcup(a_i,b_i).
\end{equation*}

Now, for each $i$, consider $\vp(a_i),\ \vp(b_i)$.
In light of Remark $\ref{actually-C1}$, we can find tangent lines $L_1, L_2$ to $C$ at these two points.
Let $T_i$ be the curvilinear triangle bounded by the three curves $L_1, L_2$ and  $C_i  = \lbrace \vp(t): a_i \leq t \leq b_i\rbrace$.
We define $K_i$ to be the union $H \cup T_i \cup (-T_i)$
as is illustrated in Diagram \eqref{tent-picture} below. 
Note that the lines $L_1, L_2$ are not parallel; one argument is that $-L_1$ is also a tangent line to $C$ by symmetry and strict convexity would prevent $L_2$ from being parallel to either of these.
Moreover, strict convexity also implies that the curve segment $C_i$ is contained in $K_i$.

We now define $K$ to be the union $\bigcup K_i$. Clearly
$K$ is  a bounded, open, convex, symmetric domain containing $H$.
Moreover, $K$ and $H$ share the boundary points $$\vp(t),\  \vq(t), \ \vq(t) - \vp(t)$$ for each $t \in Q$, which shows that the $H$-critical lattice $\phi(0)$ is $K$-admissible.

Lemma \ref{observation} thus applies to show that $\mathcal{L}(K)$ is exactly the set of $K$-admissible lattices of $\mathcal{L}(H)$.
For any $t\in Q$, the $H$-critical lattice $\phi(t)$ is $K$-admissible by Theorem \ref{threepairs}.
On the other hand, for $t\notin Q$, $\phi(t)$ is not $K$-admissible since $C_i\subset K$. This ends the proof.
\end{proof}

\begin{rem} \rm
Note the construction of above can be modified to ensure that $K$ is $\mathcal{C}^1$ by smoothing an edge of each curvilinear triangle.
\end{rem}

\begin{rem}\label{diagram-mahler} \rm Our method does not work for non-parallelogram $H$ which are not strictly convex, and it even seems unlikely that the theorem holds in that generality.
The interested reader is encouraged to examine the following irreducible domain, explicitly described in \cite[Section $12$]{ma3}. It is central to a conjecture, going back to Reinhardt \cite{re}, concerning domains $H$ that have minimal packing density $\frac{\operatorname{area}(H)}{\Delta(H)}$.

\begin{equation*}
\begin{tikzpicture}[scale=1.5, baseline=(current  bounding  box.center)]
    \draw [<->](-1.5,0) -- (1.5,0);
    \draw [<->] (0,-1.5) -- (0,1.5);
    
    \fill [red] (1,0) circle[radius=0.05];
    \fill [red] (-1,0) circle[radius=0.05];
    \fill [red] (0,0) circle[radius=0.05];
    \fill [red] (0.5,0.91421) circle[radius=0.05];
    \fill [red] (-0.5,-0.91421) circle[radius=0.05];
    \fill [red] (-0.5, 0.91421) circle[radius=0.05];
    \fill [red] (0.5, -0.91421) circle[radius=0.05];
    \draw [thick, blue] (1,-0.29289) -- (1,0.29289);
    \draw [thick, blue] (1,0.29289) to [out=90 , in=135] (0.91421,0.5);
    
    \draw [thick, blue] (0.91421,0.5) -- (0.5,0.91421);
    \draw [thick, blue] (0.5,0.91421) to [out=135 , in=180] (0.2928,1);
    
    \draw [thick, blue] (0.2928,1) -- (-0.2928,1);
    \draw [thick, blue] (-0.2928,1) to [out=180 , in=45] (-0.5, 0.91421);
    
    \draw [thick, blue] (-0.5, 0.91421) -- (-0.91421,0.5);
    \draw [thick, blue] (-0.91421,0.5) to [out=225 , in=270] (-1,0.29289);
    
    \draw [thick, blue] (-1,0.29289) -- (-1,-0.29289);
    \draw [thick, blue] (-1,-0.29289) to [out=270 , in=315] (-0.91421,-0.5);
    
    \draw [thick, blue] (-0.91421,-0.5) -- (-0.5,-0.91421);
    \draw [thick, blue] (-0.5,-0.91421) to [out=315 , in=0] (-0.2928,-1);
    
    \draw [thick, blue] (-0.2928,-1) -- (0.2928,-1);
    \draw [thick, blue] (0.2928,-1) to [out=0 , in=45] (0.5, -0.91421);
    
    \draw [thick, blue] (0.5, -0.91421) -- (0.91421,-0.5);
    \draw [thick, blue] (0.91421,-0.5) to [out=45 , in=90] (1,-0.29289);
    
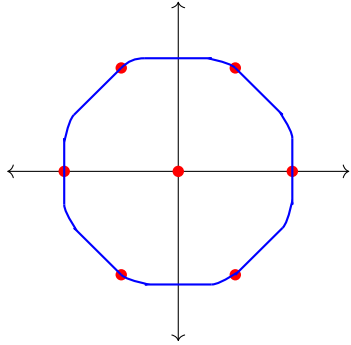
        \node [below] at (0,-1.6) {\footnotesize The curvilinear octagon with one of its critical lattices.};
    \end{tikzpicture}
    \end{equation*}
   Reinhardt proved that such a domain exists and proposed this curvilinear octagon as the unique candidate. {See  \cite{hales}
for a historical account along with some partial results.}
 \end{rem}

\section{Connections to \da} \label{da} 
Given a non-decreasing function $\psi:\R^+\to (0,1]$, say that 
a 
real number $\alpha$ is \textsl{$\psi$-Dirichlet-improvable}  if 
the system of inequalities 
\begin{equation}\label{dt}
|\alpha q-p|\le \psi(T)
\quad \textrm{and}\quad |q|\le T \end{equation}
has a nonzero solution  $(p,q)\in \Z^2$ for all sufficiently large $T$ .
Dirichlet's Theorem, see  \cite[Theorem~II.1E]{S}, asserts that the system \eqref{dt} always has a nonzero integer solution if $\psi(T) = 
1/{T
}$. On the other hand, 
it is known  \cite{Mor, DS} that the choice $\psi(T) = 
c/{T
}$ with $c<1$ produces a null set of $\psi$-Dirichlet-improvable numbers. A precise criterion for this set to have zero/full measure has been recently obtained by the first-named author and Wadleigh  \cite{KWa}. It was also shown in that paper that the property of being $\psi$-Dirichlet-improvable can be equivalently phrased in terms of dynamics on the space $X_2$ of unimodular lattices in $\R^2$. 
{Namely,  take $K$ to be the unit ball with respect to the supremum norm $\|\cdot\|$ on $\R^n$, and let $g_t := \left[\begin{matrix}
e^{ t}& 0\\
 0 & e^{-t}\end{matrix}\right]$; 
then $\alpha$ is Dirichlet-improvable if and only if  the one-parameter trajectory 
\eq{traj}{\left\{g_t\left[\begin{matrix}
1 & \alpha\\
 0 & 1\end{matrix}\right] \Z^{2}: t\ge 0\right\}\subset X_2
}
 eventually stays away from a certain family (determined by $\psi$) of shrinking neighborhoods of  the critical locus $\mathcal{L}(K)$ described in Example \ref{square}.} 
 See  \cite[Proposition 4.5]{KWa} for a precise statement, and {\cite{Da,KM, KWe} for other instances of the correspondence between \da\ and dynamics, usually referred to as the Dani Correspondence.
 
 Recently, in \cite{AD,KR} a generalization of the property of being $\psi$-Dirichlet-improvable  was introduced with the supremum norm replaced by an arbitrary norm $\nu$ on $\R^2$. Similarly to the above, that property can be restated in terms of the trajectory  \equ{traj} eventually staying away from a  family of shrinking neighborhoods of  the critical locus $\mathcal{L}(K)$, where $K$ is the unit ball with respect to $\nu$ (see  \cite[Proposition 2.1]{KR}). This has been one of the motivating factors for the study of critical loci undertaken in the present paper.
 

\section{Acknowledgements}
This collaboration was made possible by the MIT PRIMES program. 
The authors would like to commend the dedication and enthusiasm of the organizers
and thank Tanya Khovanova in particular for valuable comments and guidance.
The second named author was supported by NSF grants  DMS-1600814  and DMS-1900560.

\Addresses

\end{document}